\documentclass[12pt]{amsart}
\usepackage{amssymb}

\newtheorem{theorem}{Theorem}
\theoremstyle{plain}

\newtheorem{conjecture}[theorem]{Conjecture}

\newtheorem{proposition}[theorem]{Proposition}

\numberwithin{equation}{section}
\numberwithin{theorem}{section}
\numberwithin{exercise}{section}
\numberwithin{solution}{section}
\allowdisplaybreaks[4]
\input{tcilatex}

\begin{document}
\title[Bounded topologies on Banach spaces]{Bounded topologies on Banach
spaces and some of their uses in economic theory: a review}
\author{Andrew J. Wrobel}
\address[A. J. Wrobel, formerly of the London School of Economics]{postal
address: 15082 East County Road 600N, Charleston, Illinois, 61920-8026,
United States.}
\email{a.wrobel@alumni.lse.ac.uk}
\urladdr{https://www.researchgate.net/profile/Andrew\_Wrobel3/research}
\thanks{I am grateful to Keita Owari for the reference \cite%
{DelbaenOrihuela20}.}
\date{This version, August 30, 2020}

\begin{abstract}
Known results are reviewed about the bounded and the convex bounded
variants, $\mathrm{b}\mathcal{T}$ and $\mathrm{cb}\mathcal{T}$, of a
topology $\mathcal{T}$ on a real Banach space. The focus is on the cases of $%
\mathcal{T}\allowbreak =\mathrm{w}\left( P^{\ast },P\right) $ and of $%
\mathcal{T}\allowbreak =\mathrm{m}\left( P^{\ast },P\right) $, which are the
weak* and the Mackey topologies on a dual Banach space $P^{\ast }$. The
convex bounded Mackey topology, $\mathrm{cbm}\left( P^{\ast },P\right) $, is
known to be identical to $\mathrm{m}\left( P^{\ast },P\right) $. As for $%
\mathrm{bm}\left( P^{\ast },P\right) $, it is conjectured to be strictly
stronger than $\mathrm{m}\left( P^{\ast },P\right) $ or, equivalently, 
\textsl{not\/} to be a vector topology (except when $P$ is reflexive). Some
uses of the bounded Mackey and the bounded weak* topologies in economic
theory and its applications are pointed to. Also reviewed are the bounded
weak and the compact weak topologies, $\mathrm{bw}\left( Y,Y^{\ast }\right) $
and $\mathrm{kw}\left( Y,Y^{\ast }\right) $, on a general Banach space $Y$,
as well as their convex variants ($\mathrm{cbw}$ and $\mathrm{ckw}$).
\end{abstract}

\maketitle

\section{Introduction}

Nonmetric topologies on the norm-dual, $P^{\ast }$, of a real Banach space ($%
P$) can become much more manageable when restricted to bounded sets. For
example, given a convex subset of $P^{\ast }$, or a real-valued concave
function on $P^{\ast }$, the bounded weak* topology, $\mathrm{bw}^{\ast
}\allowbreak :=\mathrm{bw}\left( P^{\ast },P\right) $, can serve to show
that the set in question is weakly* closed, or that the function is weakly*
upper semicontinuous. In economic theory, such uses of the Krein-Smulian
Theorem are made in \cite[Proposition 1.1, Theorems 4.4 and 4.7]%
{DelbaenOwari19}, \cite[Proposition 1 and Example 5]{H-WLocCont}, %
\cite[Lemma 4.1]{H-WDemContEqEx} and \cite[Section 6.2]{H-W-SRA}. In
applications of economic equilibrium models, this can be an indispensable
tool for verifying that the production sets that describe the technologies
are weakly* closed, and that the profit and cost functions are weakly*
semicontinuous (which is needed for equilibria to exist, and for the dual
pairs of programmes to have no duality gaps): see \cite[Lemma 17.1]%
{H-WPumSto}, \cite[Lemma 6.1]{H-WStoHyd} and \cite[Lemmas 6.2.3--6.2.5]%
{H-W-SRA}.

When $P$ is $L^{1}\left( T,\sigma \right) $, the space of integrable
real-valued functions on a set $T$ that carries a sigma-finite measure $%
\sigma $---and so $P^{\ast }$ is the space of essentially bounded functions $%
L^{\infty }\left( T\right) $---another useful ``bounded'' topology on $%
L^{\infty }$ is the bounded Mackey topology, $\mathrm{bm}\left( L^{\infty
},L^{1}\right) $. This is because a concave real-valued function, $F$, is
continuous for the ``plain'' Mackey topology, $\mathrm{m}\left( L^{\infty
},L^{1}\right) $, if (and only if) it is $\mathrm{bm}\left( L^{\infty
},L^{1}\right) $-continuous, i.e., $\mathrm{m}\left( L^{\infty
},L^{1}\right) $-continuous on bounded sets---or, equivalently, if (and only
if) $F$ is continuous along bounded sequences (in $L^{\infty }$) that
converge in measure (on subsets of $T$ of finite measure).\footnote{%
Continuity along such sequences is also known as the Lebesgue property (of $%
F $): see, e.g., \cite[Definition 1.2]{Jouini-Schach-Touzi06}. It is
equivalent to $\mathrm{bm}\left( L^{\infty },L^{1}\right) $-continuity
because the topology of convergence in measure (on sets of finite measure), $%
\mathcal{T}_{\sigma }$, is both metrizable (on $L^{\infty }$) and equal to $%
\mathrm{m}\left( L^{\infty },L^{1}\right) $ on bounded subsets of $L^{\infty
}$: see, e.g., \cite[Example 8.47 (3)]{Aliprantis-BorderIDA06} and \cite[pp.
222--223]{GrothendieckTVS}. The metric in \cite[Example 8.47 (3)]%
{Aliprantis-BorderIDA06} is for the case of $\sigma \left( T\right)
\allowbreak <+\infty $, but $\mathcal{T}_{\sigma }$ is metrizable also when $%
\sigma $ is sigma-finite. Also, on $L^{\infty }$ globally, $\mathcal{T}%
_{\sigma }$ is weaker than $\mathrm{m}\left( L^{\infty },L^{1}\right) $.}
Thus the reduction to bounded sets provides direct access to the methods of
integral calculus, which can greatly simplify verification of Mackey
continuity \cite[Example 5]{H-WLocCont}. And, in economic equilibrium
analysis, Mackey continuity of a concave utility or production function ($F$%
) is essential for representing the price system by a density, as is done in %
\cite{Bewley72} and \cite{H-WPriceDens}. In addition, the use of convergence
in measure furnishes economic interpretations of Mackey continuity %
\cite[Sections 4 and 5]{H-WLocCont}. (Thus it also makes clear the
restrictiveness of this condition and of the resulting density form of the
price system, which excludes the singularities that alone can represent
capital charges when these are extremely concentrated in time or space. The
alternative is not to exclude the ``intractable'' singular functionals but
to re-represent them \cite{Wrobel18-Singul}.)

In the case that the concave function $F$ is defined and finite on the 
\textsl{whole\/} space $L^{\infty }$, the equivalence of $\mathrm{m}\left(
L^{\infty },L^{1}\right) $-continuity to $\mathrm{bm}\left( L^{\infty
},L^{1}\right) $-continuity can be shown by using the Fenchel-Legendre
conjugacy. This is a result of Delbaen and Owari \cite[Proposition 1.2]%
{DelbaenOwari19}, who also extend it to the case of a dual Orlicz space
instead of $L^{\infty }$ \cite[Theorem 4.5]{DelbaenOwari19} and apply it in
the mathematics of finance \cite[Theorem 4.8]{DelbaenOwari19}. Their
argument shows first that if $F$ is $\mathrm{bm}\left( L^{\infty
},L^{1}\right) $-continuous then the superlevel sets of its concave
conjugate (a function on $L^{1}$) are uniformly integrable, and then applies
the Dunford-Pettis Compactness Criterion and the Moreau-Rockafellar Theorem
(on the conjugacy between continuity and sup-compactness); the first step is
made also in \cite[Theorem 5.2 (i) and (iv)]{Jouini-Schach-Touzi06}.

The case of a nondecreasing concave $F$ that is defined only on the
nonnegative cone $L_{+}^{\infty }$ (and does not have a finite concave
extension to $L^{\infty }$) requires a different method: it relies on Mackey
continuity of the lattice operations in $L^{\infty }$, as well as on the
monotonicity of $F$ \cite[Proposition 3 and Example 4]{H-WLocCont}.

For a finite-valued concave $F$ defined on the \textsl{whole\/} space, the
equivalence of Mackey continuity to bounded Mackey continuity extends to the
case of a general dual Banach space, $P^{\ast }$, as the domain of $F$: see
Delbaen and Orihuela \cite[Theorem 8]{DelbaenOrihuela20}. (Equivalence to
sequential Mackey continuity follows when $P$ is strongly weakly compactly
generated \cite[Corollary 11]{DelbaenOrihuela20}.) \textsl{A fortiori\/},
those linear functionals (on $P^{\ast }$) that are continuous for the
bounded Mackey topology, $\mathrm{bm}^{\ast }\allowbreak :=\mathrm{bm}\left(
P^{\ast },P\right) $, are actually continuous for $\mathrm{m}^{\ast
}\allowbreak :=\mathrm{m}\left( P^{\ast },P\right) $, i.e., belong to $P$.
It follows that the \textsl{convex\/} bounded Mackey topology, $\mathrm{cbm}%
\left( P^{\ast },P\right) $, is identical to the ``plain'' Mackey topology:%
\footnote{%
For $P\allowbreak =L^{1}$ only, that $\mathrm{cbm}\left( L^{\infty
},L^{1}\right) \allowbreak =\mathrm{m}\left( L^{\infty },L^{1}\right) $ has
been shown earlier by methods specific to this space, in \cite[III.1.6 and
III.1.9]{CooperSSAFA} and in \cite[Theorem 5]{Nowak89}.} $\mathrm{cbm}^{\ast
}\allowbreak =\mathrm{m}^{\ast }$ for every $P$.\footnote{%
It also follows that $\mathrm{bm}^{\ast }$-continuity upgrades to $\mathrm{m}%
^{\ast }$-continuity not only for linear functionals but also for \textsl{%
general\/} linear maps, i.e., every $\mathrm{bm}^{\ast }$-continuous linear
map of $P^{\ast }$, into \textsl{any\/} topological vector space, is $%
\mathrm{m}^{\ast }$-continuous (on $P^{\ast }$). This is because, for a 
\textsl{linear\/} map of a space with topologies of the forms $\mathrm{b}%
\mathcal{T}$ and $\mathrm{cb}\mathcal{T}$, its $\mathrm{b}\mathcal{T}$%
-continuity implies $\mathrm{cb}\mathcal{T}$-continuity \cite[I.1.7]%
{CooperSSAFA}, and because $\mathrm{cbm}^{\ast }\allowbreak =\mathrm{m}%
^{\ast }$.} Implicit in \cite[Proposition 1]{H-WLocCont}, this result is
derived more simply from Grothen\-dieck's Completeness Theorem; it is quoted
here from \cite{Wrobel20MckContDualSp} as Proposition~\ref{m*=cbm*}. It does 
\textsl{not\/} follow that $\mathrm{bm}^{\ast }$ equals $\mathrm{m}^{\ast }$
because $\mathrm{bm}^{\ast }$ is not known to be a vector topology and,
indeed, it is conjectured not to be one (unless $P$ is reflexive): see \cite%
{Wrobel20MckContDualSp} or Conjecture~\ref{BddMack*UneqMack*} here.

As for the convex bounded variant of the weak* topology, $\mathrm{cbw}^{\ast
}$ is the same as $\mathrm{bw}^{\ast }$ (since the latter is locally convex
by the Banach-Dieudonn\'{e} Theorem), and so it is strictly stronger than
the ``plain'' $\mathrm{w}^{\ast }$ (unless $P$ is finite-dimensional).

Every Banach space, $Y$, whether dual or not, carries also the bounded weak
topology and its convex variant, $\mathrm{bw}$ and $\mathrm{cbw}$ (which
differ from each other unless $Y$ is reflexive, in which case $\mathrm{bw}%
^{\ast }\allowbreak =\mathrm{bw}\allowbreak =\mathrm{cbw}$). Studied in \cite%
{GomezGil84} and \cite{Wheeler72}, $\mathrm{bw}$ and $\mathrm{cbw}$ are
briefly discussed at the end of Section~\ref{BddConvBddTopols}. The space ($%
Y $) carries also the compact weak topology and its convex variant, $\mathrm{%
kw}$ and $\mathrm{ckw}$. Introduced in \cite{GonzalezGutierrez92}, this
concept produces a new topology (or two) if and only if $Y$ contains the
sequence space $l^{1}$ (if it does not, then $\mathrm{kw}\allowbreak =%
\mathrm{bw}$ and so $\mathrm{ckw}\allowbreak =\mathrm{cbw}$ too): see
Section~\ref{CompConvCompWeakTopols}. These topologies are used in studying
function spaces and linear operations: see \cite{FerreraThesisEFDC} and %
\cite[Chapter 4]{LlavonaACDF} for such uses of $\mathrm{bw}$ and $\mathrm{cbw%
}$, and \cite{GonzalezGutierrez92} and \cite{GonzalezGutierrez93} for those
of $\mathrm{kw}$ and $\mathrm{ckw}$.

\section{The bounded and convex bounded topologies}

\label{BddConvBddTopols}The weakest and the strongest of those locally
convex topologies on a dual Banach space $P^{\ast }$ which yield $P$ as the
continuous dual are denoted by $\mathrm{w}\left( P^{\ast },P\right) $ and $%
\mathrm{m}\left( P^{\ast },P\right) $, abbreviated to $\mathrm{w}^{\ast }$
and $\mathrm{m}^{\ast }$. Known as the \textsl{weak\/} and the \textsl{%
Mackey\/} topologies, on $P^{\ast }$ for its pairing with $P$, the two can
be called the weak* and the Mackey topologies (since the other Mackey
topology on $P^{\ast }$, $\mathrm{m}\left( P^{\ast },P^{\ast \ast }\right) $%
, is identical to the norm topology). The \textsl{bounded weak*\/} topology
on $P^{\ast }$ is denoted by $\mathrm{bw}\left( P^{\ast },P\right) $,
abbreviated to $\mathrm{bw}^{\ast }$. It can be defined by stipulating that
a subset of $P^{\ast }$ is $\mathrm{bw}^{\ast }$-closed if and only if its
intersection with every closed ball in $P^{\ast }$ is $\mathrm{w}^{\ast }$%
-closed (or, equivalently, $\mathrm{w}^{\ast }$-compact). In other words, $%
\mathrm{bw}^{\ast }$ is the strongest topology that is equal to the weak*
topology on every bounded subset (of $P^{\ast }$). Directly from its
definition, $\mathrm{bw}^{\ast }$ is stronger than $\mathrm{w}^{\ast }$ (and
is strictly so unless $P$ is finite-dimensional). The Banach-Dieudonn\'{e}
Theorem identifies $\mathrm{bw}^{\ast }$ as the topology of uniform
convergence on norm-compact subsets of $P$: see, e.g., \cite[p. 159: Theorem
2]{GrothendieckTVS}, \cite[18D: Corollary (b)]{HolmesGFA} or \cite[IV.6.3:
Corollary 2]{SchaeferTVS}. It follows that: (i)~$\mathrm{bw}^{\ast }$ is
locally convex, and (ii)~$\mathrm{bw}^{\ast }$ is weaker than $\mathrm{m}%
^{\ast }$.\footnote{%
To deduce that $\mathrm{bw}^{\ast }$ is weaker than $\mathrm{m}^{\ast }$,
recall that $\mathrm{m}\left( P^{\ast },P\right) $ is the topology of
uniform convergence on all weakly compact subsets of $P$. See, e.g., %
\cite[Section 5.18]{Aliprantis-BorderIDA06} or \cite[IV.3.2: Corollary 1]%
{SchaeferTVS}, where the compacts are required to be convex and circled as
well, but here ``convex'' can be omitted because $P$ is a Banach space and
one can apply Krein's Theorem---for which see, e.g., \cite[Theorem 6.35,
named Krein-Smulian]{Aliprantis-BorderIDA06}, \cite[V.6.4]%
{Dunford-SchwartzLOI}, \cite[19E]{HolmesGFA} or \cite[IV.11.4]{SchaeferTVS}.
That ``circled'' can be omitted is obvious \cite[I.5.2]{SchaeferTVS}.} So,
since every \textsl{convex\/} $\mathrm{m}^{\ast }$-closed set is $\mathrm{w}%
^{\ast }$-closed, it follows that every convex $\mathrm{bw}^{\ast }$-closed
set is $\mathrm{w}^{\ast }$-closed; this the Krein-Smulian Theorem, for
which see, e.g., \cite[V.5.7]{Dunford-SchwartzLOI}, \cite[18E: Corollary 2]%
{HolmesGFA} or \cite[IV.6.4]{SchaeferTVS}. Also, given that $\mathrm{bw}%
^{\ast }$ is locally convex by Part~(i), Part~(ii)---that $\mathrm{bw}^{\ast
}\allowbreak \subseteq \mathrm{m}^{\ast }$---can be restated as: the $%
\mathrm{bw}^{\ast }$-continuous dual of $P^{\ast }$ equals $P$. This
equality requires the norm (of $P$) to be complete. Indeed, it is a special
case of Grothen\-dieck's Completeness Theorem; for this case see, e.g., %
\cite[V.5.5 and V.5.6]{Dunford-SchwartzLOI}, \cite[18E: Corollary 1]%
{HolmesGFA} or \cite[IV.6.2: Corollary 2]{SchaeferTVS}.\footnote{%
To prove that the $\mathrm{bw}^{\ast }$-dual of $P^{\ast }$ is $P$ from the
standard formulation of Grothendieck's Theorem (Theorem~\ref{GrothendieckThm}
here), apply it to $P^{\ast }$ as $E$---with $\mathrm{w}^{\ast }$ as $%
\mathfrak{T}$ and the bounded subsets of $P^{\ast }$ as $\mathfrak{S}$, and
hence with $P$ as $E^{\prime }$ and the norm topology of $P$ as the $%
\mathfrak{S}$-topology---to conclude that a linear functional on $P^{\ast }$
is $\mathrm{w}^{\ast }$-continuous if it is so on bounded sets, i.e., if it
is $\mathrm{bw}^{\ast }$-continuous.}

By its definition, $\mathrm{bw}^{\ast }$ is a case of the general concept of
``bounding'' a locally convex topology $\mathcal{T}$, on a space $Y$ with a
norm $\left\| \cdot \right\| $, to produce the strongest topology that is
equal to $\mathcal{T}$ on every norm-bounded subset (of $Y$). It is assumed
that: (i)~$\mathcal{T}$ is weaker than the norm topology, and (ii)~the
closed unit ball of $Y$ is $\mathcal{T}$-closed; such a $\mathcal{T}$ is
said to be \textsl{compatible\/} with the norm (of $Y$), and $\left(
Y,\left\| \cdot \right\| ,\mathcal{T}\right) $ is then called a \textsl{Saks
space\/} \cite[p. 6]{CooperSSAFA}.\footnote{%
And so $\mathcal{T}$ is compatible with the norm of $Y$ if $\mathcal{T}$ is
weaker than the norm topology but stronger than $\mathrm{w}\left( Y,Y^{\ast
}\right) $, or at least stronger than $\mathrm{w}\left( Y,P\right) $ when $%
Y\allowbreak =P^{\ast }$.} The resulting \textsl{bounded }$\mathcal{T}$%
\textsl{-topology\/}, denoted here by $\mathrm{b}\mathcal{T}$, is stronger
than $\mathcal{T}$ (and weaker than the norm). Put in other words, a subset
of $Y$ is $\mathrm{b}\mathcal{T}$-closed if and only if its intersection
with every closed ball of $Y$ is $\mathcal{T}$-closed. By \cite[Theorem 5]%
{Collins55}---as is noted also in \cite[2.7]{FerreraThesisEFDC}, \cite[p.
410]{GildeLamadrid59} and \cite[p. 72]{GomezGil84}---$\mathrm{b}\mathcal{T}$
is always \textsl{semi\/}-linear (i.e., both vector addition and scalar
multiplication are \textsl{separately\/} continuous in either variable), but
generally it need \textsl{not\/} be linear (although $\mathrm{bw}^{\ast }$
is). A map of $Y$ (into a topological space) is $\mathrm{b}\mathcal{T}$%
-continuous if and only if its restrictions to bounded sets are $\mathcal{T}$%
-continuous \cite[Theorem 1 (b)]{GildeLamadrid59}.

When $\left( Y,\left\| \cdot \right\| ,\mathcal{T}\right) $ is a Saks space,
the norm-compatible topology $\mathcal{T}$ can be ``mixed'' with the norm to
produce the strongest \textsl{vector\/} topology that is equal to $\mathcal{T%
}$ on every norm-bounded subset (of $Y$). Remarkably, this is also the
strongest \textsl{locally convex\/} topology that is equal to $\mathcal{T}$
on every bounded set: this is shown in \cite[I.1.4 and I.1.5 (iii)]%
{CooperSSAFA} and \cite[2.2.2]{Wiweger61}, and is stated also in \cite[1.39
and 1.40]{FerreraThesisEFDC}---where the resulting topology is denoted by $%
\gamma \left( \left\| \cdot \right\| ,\mathcal{T}\right) $, or by $\gamma
\left( \mathcal{B},\mathcal{T}\right) $ with $\mathcal{B}$ for the bounded
sets. Here, this \textsl{convex bounded }$\mathcal{T}$\textsl{-topology\/}
is denoted by $\mathrm{cb}\mathcal{T}$; it is stronger than $\mathcal{T}$
and weaker than $\mathrm{b}\mathcal{T}$.\footnote{%
Obviously, $\mathrm{cb}\mathcal{T}\allowbreak =\mathrm{b}\mathcal{T}$ if and
only if $\mathrm{b}\mathcal{T}$ is locally convex (or, equivalently, is a
vector topology).} (The three are, however, \textsl{sequentially\/}
equivalent (i.e., have the same convergent sequences) when $Y\allowbreak
=P^{\ast }$ is a dual Banach space and $\mathcal{T}$ is stronger than $%
\mathrm{w}^{\ast }$ (e.g., when $\mathcal{T}$ is $\mathrm{m}^{\ast }$ or $%
\mathrm{w}^{\ast }$ itself). This is because, unlike a general uncountable
net, a $\mathcal{T}$-convergent sequence, being $\mathrm{w}^{\ast }$%
-convergent, is bounded by the Banach-Steinhaus Theorem---and so it is $%
\mathrm{b}\mathcal{T}$-convergent.) A \textsl{linear\/} map of $Y$ (into a
topological vector space) is $\mathrm{cb}\mathcal{T}$-continuous if (and
only if) it is $\mathrm{b}\mathcal{T}$-continuous, i.e., if (and only if)
its restrictions to bounded sets are $\mathcal{T}$-continuous \cite[I.1.7]%
{CooperSSAFA}. As a case of this, a linear functional on $Y$ is $\mathrm{cb}%
\mathcal{T}$-continuous if (and only if) it is $\mathrm{b}\mathcal{T}$%
-continuous (but, to avoid misapplying this, recall that $\mathrm{b}\mathcal{%
T}$ need not be a vector topology).

For $Y\allowbreak =P^{\ast }$ with $\mathcal{T}\allowbreak =\mathrm{w}^{\ast
}\allowbreak :=\mathrm{w}\left( P^{\ast },P\right) $, where $P$ is a (real)
Banach space, the bounded weak* topology is itself locally convex, and so it
is identical to its convex variant: $\mathrm{cbw}^{\ast }\allowbreak =%
\mathrm{bw}^{\ast }$. The case of the \textsl{convex bounded Mackey\/}
topology is different: $\mathrm{cbm}^{\ast }\allowbreak =\mathrm{m}^{\ast }$
(on the whole space $P^{\ast }$). As is set out next, this follows from
Grothen\-dieck's Completeness Theorem \cite[IV.6.2]{SchaeferTVS}, which is
quoted for easy reference.\footnote{%
Grothen\-dieck's Completeness Theorem is also stated in \cite[p. 73:
Corollary 1]{GrothendieckTVS}.}

\begin{theorem}[Grothendieck]
\label{GrothendieckThm}Let $\mathfrak{T}$ be a locally convex topology on a
real vector space $E$. When additionally $\mathfrak{S}$ is a saturated family%
\footnote{%
A family, $\mathfrak{S}$, of subsets of a locally convex space is called 
\textsl{saturated\/} \cite[p. 81]{SchaeferTVS} if: (i)~all subsets of every
member of $\mathfrak{S}$ belong to $\mathfrak{S}$, (ii)~all scalar multiples
of every member of $\mathfrak{S}$ belong to $\mathfrak{S}$, and (iii)~for
each finite $\mathfrak{F\allowbreak }\subset \mathfrak{S}$, the closed
convex circled hull of the union of $\mathfrak{F}$ belongs to $\mathfrak{S}$.%
} of $\mathfrak{T}$-bounded sets covering $E$, the $\mathfrak{T}$-dual of $E$
is complete under the $\mathfrak{S}$-topology (the topology of uniform
convergence on every $S\allowbreak \in \mathfrak{S}$) if and \textsl{only\/}
if every linear functional (on $E$) that is $\mathfrak{T}$-continuous on
each $S\allowbreak \in \mathfrak{S}$ is actually $\mathfrak{T}$-continuous
on the whole space $E$ (i.e., is in the $\mathfrak{T}$-dual of $E$).
\end{theorem}

\begin{proposition}[\protect\cite{Wrobel20MckContDualSp}]
\label{m*=cbm*}Let $P$ be a real Banach space, and $P^{\ast }$ its
norm-dual. Then $\mathrm{cbm}\left( P^{\ast },P\right) \allowbreak =\mathrm{m%
}\left( P^{\ast },P\right) $.
\end{proposition}

\begin{proof}
Apply Theorem~\ref{GrothendieckThm} to $P^{\ast }$ as $E$---with $\mathrm{m}%
^{\ast }$ as $\mathfrak{T}$ and the bounded subsets of $P^{\ast }$ as $%
\mathfrak{S}$, and hence with $P$ as the $\mathfrak{T}$-dual and the norm
topology of $P$ as the $\mathfrak{S}$-topology---to conclude that a linear
functional on $P^{\ast }$ is $\mathrm{m}^{\ast }$-continuous if it is so on
bounded sets (i.e., if it is $\mathrm{bm}^{\ast }$-continuous). \textsl{A
fortiori\/}, it is $\mathrm{m}^{\ast }$-continuous (i.e., is in $P$) if it
is $\mathrm{cbm}^{\ast }$-continuous. In other words, $\mathrm{cbm}^{\ast }$
yields the same dual space as $\mathrm{m}^{\ast }$ (viz., $P$). This proves
that $\mathrm{cbm}^{\ast }\allowbreak =\mathrm{m}^{\ast }$ (since $\mathrm{%
cbm}^{\ast }$ is both locally convex and stronger than $\mathrm{m}^{\ast }$).%
\footnote{%
Alternatively, Cooper's special case of Grothen\-dieck's Theorem %
\cite[I.1.17 (ii)]{CooperSSAFA} can be applied---to $P^{\ast }$ as his $E$,
with $\mathrm{m}^{\ast }$ as $\tau $ and the bounded subsets of $P^{\ast }$
as $\mathcal{B}$, and hence with $P^{\ast \ast }$ as $E_{\mathcal{B}%
}^{\prime }$ and $\mathrm{cbm}^{\ast }$ as his $\gamma \allowbreak =\gamma
\left( \mathcal{B},\tau \right) $---to conclude that the $\mathrm{cbm}^{\ast
}$-dual equals the $\mathrm{m}^{\ast }$-dual (so $\mathrm{cbm}^{\ast
}\allowbreak =\mathrm{m}^{\ast }$).}
\end{proof}

As for $\mathrm{bm}^{\ast }$, it is \textsl{conjectured\/} to be different
from $\mathrm{m}^{\ast }$.

\begin{conjecture}[\protect\cite{Wrobel20MckContDualSp}]
\label{BddMack*UneqMack*}For $P=L^{1}\left[ 0,1\right] $ at least, and
possibly for every nonreflexive Banach space $P$, the topology $\mathrm{bm}%
\left( P^{\ast },P\right) $ is strictly stronger than $\mathrm{m}\left(
P^{\ast },P\right) $---or, equivalently, $\mathrm{bm}\left( P^{\ast
},P\right) $ is not linear.
\end{conjecture}

This conjecture is based on what it takes to establish that a $\mathrm{bm}%
^{\ast }$-continuous $\mathbb{R}$-valued function $F$, on a nonreflexive
space $P^{\ast }$, is $\mathrm{m}^{\ast }$-continuous: the
Delbaen-Orihuela-Owari results of \cite[Theorem 8]{DelbaenOrihuela20} and %
\cite[Proposition 1.2 and Theorem 4.5]{DelbaenOwari19} require $F$ to be
concave (or convex), and it is hard to imagine (even when $P^{\ast
}\allowbreak =L^{\infty }\left[ 0,1\right] $) how the convexity assumption
might be disposed of entirely---as would be necessary for $\mathrm{bm}^{\ast
}$ to equal $\mathrm{m}^{\ast }$.\footnote{%
Note also that the sufficient condition of \cite[I.4.2]{CooperSSAFA} for $%
\mathrm{b}\mathcal{T}$ to equal $\mathrm{cb}\mathcal{T}$ does not apply to $%
\mathcal{T}\allowbreak =\mathrm{m}^{\ast }$ (since it means that $\mathcal{T}
$ is the weak* topology on the dual of a Fr\'{e}chet space \cite[I.4.1 and
I.2.A]{CooperSSAFA}).}\medskip

\textbf{Comments} (completeness of $\mathrm{m}\left( P^{\ast },P\right) $
and lattice properties of $\mathrm{m}\left( L^{\infty },L^{1}\right) $):

\begin{itemize}
\item As is observed in, e.g., \cite[pp. 97--98]{GuiraoMontesinos15FAA} and %
\cite[1.1]{SchluchtermannWheeler88}, $\mathrm{m}\left( P^{\ast },P\right) $
is complete. This is a different application of Grothen\-dieck's
Theorem---one that swaps the spaces' roles and works in the ``other
direction'' to \textsl{prove\/} completeness (of $\mathrm{m}^{\ast }$, on $%
P^{\ast }$), rather than \textsl{using\/} completeness (of the norm on $P$)
as here (to prove that $\mathrm{cbm}^{\ast }\allowbreak =\mathrm{m}^{\ast }$%
). And although the argument of \cite{GuiraoMontesinos15FAA} and \cite%
{SchluchtermannWheeler88} does use the norm-completeness of $P$, this is
needed only in its first step, which uses Krein's Theorem \cite[19E]%
{HolmesGFA} rather than Grothen\-dieck's.

\item In \cite[2.1]{SchluchtermannWheeler88} it is also shown that $\mathrm{m%
}\left( P^{\ast },P\right) $ is (completely) metrizable on bounded sets if
and only if $P$ is strongly weakly compactly generated (SWCG).

\item For $P\allowbreak =L^{1}$, it follows from the Dunford-Pettis
Compactness Criterion that $\mathrm{m}\left( L^{\infty },L^{1}\right) $ is a
Lebesgue topology, i.e., it is (i)~locally solid (that is, it makes $%
L^{\infty }$ a topological vector lattice), and (ii)~order-continuous: see,
respectively, \cite[p. 535]{Bewley72} or \cite[Theorem 9.36]%
{Aliprantis-BorderIDA06} or \cite[Chapter 6: Exercise 4]%
{Aliprantis-BurkinshawLSRS}, and \cite[9.1, equivalence of (i) and (ii)]%
{Aliprantis-BurkinshawLSRS} or \cite[3.12, equivalence of (1) and (2)]%
{Aliprantis-BurkinshawLSRSwAE}. In \cite[Theorems 4 and 5]{Nowak89}, $%
\mathrm{m}\left( L^{\infty },L^{1}\right) $ is shown to be the strongest
Lebesgue topology on $L^{\infty }$, by an argument which also shows that $%
\mathrm{m}\left( L^{\infty },L^{1}\right) \allowbreak =\mathrm{cbm}\left(
L^{\infty },L^{1}\right) $.\footnote{%
In detail: the mixed topology $\gamma \allowbreak =\gamma \left( \left\|
\cdot \right\| _{\infty },\mathcal{T}_{\sigma }\right) $---where $\mathcal{T}%
_{\sigma }$ is the topology of convergence in measure (on sets of finite
measure)---is shown in \cite[Theorems 2 and 4]{Nowak89} to yield $%
L^{1}\left( \sigma \right) $ as its dual, and to be the strongest Lebesgue
topology on $L^{\infty }$ (which is of interest in itself). It is then
deduced \cite[Theorem 5]{Nowak89} that $\mathrm{m}\left( L^{\infty
},L^{1}\right) \allowbreak =\gamma $. But $\gamma \allowbreak =\mathrm{cbm}%
\left( L^{\infty },L^{1}\right) $ because $\mathcal{T}_{\sigma }$ is equal
to $\mathrm{m}\left( L^{\infty },L^{1}\right) $ on bounded sets, as is seen
from the Dunford-Pettis Criterion \cite[pp. 222--223]{GrothendieckTVS}.} In %
\cite[Theorem 6]{Nowak89}, the Lebesgue property of $\mathrm{m}\left(
L^{\infty },L^{1}\right) $ is used to show that this topology is also
complete---but, like the equality $\mathrm{cbm}^{\ast }\allowbreak =\mathrm{m%
}^{\ast }$, this too is actually true of $\mathrm{m}^{\ast }$ for every
Banach space $P$ (whether ordered or not).

\item The $\mathrm{m}\left( L^{\infty },L^{1}\right) $-continuity of lattice
operations has various uses in economic theory, such as those in \cite%
{H-WFreeDisp}, \cite[Example 4]{H-WLocCont} and \cite[Proof of Theorem 15]%
{H-WPriceDens}, in addition to those mentioned in \cite[p. 361]%
{Aliprantis-BorderIDA06}.\medskip
\end{itemize}

\textbf{Comments} (on $\mathrm{bm}\left( Y,Y^{\ast }\right) $, $\mathrm{bm}%
\left( P^{\ast },P\right) $ and $\mathrm{bw}\left( Y,Y^{\ast }\right) $):

\begin{itemize}
\item As for $\mathrm{bm}\left( Y,Y^{\ast }\right) $, where $Y$ is any
Banach space, it is of course the norm topology of $Y$, since $\mathrm{m}%
\left( Y,Y^{\ast }\right) $ is.

\item When $P$ is reflexive, setting $Y\allowbreak =P^{\ast }$ above (with $%
Y^{\ast }\allowbreak =P^{\ast \ast }\allowbreak =P$) shows that $\mathrm{bm}%
\left( P^{\ast },P\right) \allowbreak =\mathrm{m}\left( P^{\ast },P\right) $%
, and that it then is the norm topology of $P^{\ast }$.

\item In no space can $\mathrm{bm}^{\ast }$ be both linear \textsl{and\/}
different from $\mathrm{m}^{\ast }$. This is in contrast to the case of $%
\mathrm{bw}^{\ast }$ and $\mathrm{w}^{\ast }$ (but this is not strange
because the ``bounding'' strengthens the topology, and in the cases of $%
\mathrm{w}^{\ast }$ and $\mathrm{m}^{\ast }$ it starts at the opposite
extremes, in strength, of the range of the locally convex topologies on $%
P^{\ast }$ for its pairing with $P$).

\item ``Bounding'' the weak topology (when it is not $\mathrm{w}^{\ast }$)
fails to produce a linear one. That is, the \textsl{bounded weak\/}
topology, $\mathrm{bw}\left( Y,Y^{\ast }\right) $ or $\mathrm{bw}$ for
brevity, on a Banach space $Y$ is \textsl{not\/} locally convex (except, of
course, when $Y$ is reflexive, in which case $Y^{\ast }$ is also the unique
norm-predual of $Y$, and $\mathrm{bw}\allowbreak =\mathrm{bw}^{\ast }$): see %
\cite[3.7]{GomezGil84} or \cite[4.2.8]{LlavonaACDF}. It follows---as a case
of \cite[I.1.4 and I.1.5 (iii)]{CooperSSAFA} that is noted also in \cite[2.5]%
{FerreraThesisEFDC} and \cite[p. 72]{GomezGil84}---that $\mathrm{bw}$ is not
even linear (unless $Y$ is reflexive). In other words, $\mathrm{bw}$ is 
\textsl{strictly\/} stronger than $\mathrm{cbw}$, the \textsl{convex bounded
weak\/} topology. And it is \textsl{not\/} $\mathrm{bw}\left( Y,Y^{\ast
}\right) $ but $\mathrm{cbw}\left( Y,Y^{\ast }\right) $ that equals the
restriction (to $Y$) of $\mathrm{bw}\left( Y^{\ast \ast },Y^{\ast }\right) $%
, the bounded weak* topology of the second norm-dual $Y^{\ast \ast }$ (when
it differs from $Y$): see \cite[2.6]{GomezGil84}. For $Y\allowbreak =c_{0}$,
the space of real sequences converging to zero, an example of a $\mathrm{bw}$%
-closed set that is not $\mathrm{cbw}$-closed is given in \cite[4.8]%
{Wheeler72}; it is reproduced in \cite[p. 48, II.5(2)(b)]%
{Day73-NormLinSp(3rd)} and \cite[2.1]{FerreraThesisEFDC}. In \cite[3.1]%
{GomezGil84}, this example is generalized to any separable nonreflexive $Y$
that is sequentially reflexive or, equivalently by \cite{Orno91}, does not
contain an isomorphic copy of $l^{1}$ (the space of summable sequences).
\end{itemize}

\section{A summary of comparisons of $\mathcal{T}$, $\mathrm{cb}\mathcal{T}$
and $\mathrm{b}\mathcal{T}$ for $\mathcal{T}=\mathrm{w}^{\ast }$, $\mathrm{m}%
^{\ast }$, $\mathrm{w}$}

\label{SummComparBddTops}Except for the one which is only conjectured, the
following strict inclusions and equalities hold (for the topologies as
families of open/closed sets):

\begin{itemize}
\item $\mathrm{w}^{\ast }\varsubsetneq \mathrm{cbw}^{\ast }=\mathrm{bw}%
^{\ast }$. The equality holds by the Banach-Dieudonn\'{e} Theorem; the
inclusion is strict for all infinite-dimensional Banach spaces \cite[p. 48,
II.5(2)(a)]{Day73-NormLinSp(3rd)}.

\item $\mathrm{m}^{\ast }=\mathrm{cbm}^{\ast }$ (Proposition~\ref{m*=cbm*}).
Is $\mathrm{m}^{\ast }\varsubsetneq \mathrm{bm}^{\ast }$? (Conjecture~\ref%
{BddMack*UneqMack*}, for nonreflexive spaces.)

\item $\mathrm{w}\varsubsetneq \mathrm{cbw}\varsubsetneq \mathrm{bw}$. That
the second inclusion is strict (unless the space is reflexive and so $%
\mathrm{w}\allowbreak =\mathrm{w}^{\ast }$) is shown in \cite[2.6]%
{GomezGil84}.
\end{itemize}

\section{The compact weak and convex compact weak topologies}

\label{CompConvCompWeakTopols}Replacing the bounded sets in the definition
of $\mathrm{bw}$ by weak compacts produces the \textsl{compact weak\/}
topology, $\mathrm{kw}\left( Y,Y^{\ast }\right) $ or $\mathrm{kw}$ for
brevity, on a Banach space $Y$ (paired with its norm-dual $Y^{\ast }$).
Introduced in \cite{GonzalezGutierrez92}, $\mathrm{kw}\left( Y,Y^{\ast
}\right) $ is, then, defined as the strongest topology that is equal to $%
\mathrm{w}\left( Y,Y^{\ast }\right) $ on every $\mathrm{w}\left( Y,Y^{\ast
}\right) $-compact subset \cite[2.3 (b)]{GonzalezGutierrez92}. In other
words, a subset of $Y$ is $\mathrm{kw}$-closed if and only if its
intersection with every $\mathrm{w}$-compact set is $\mathrm{w}$-closed or,
equivalently, $\mathrm{w}$-compact \cite[2.1]{GonzalezGutierrez92}. An
equivalent characterization is that $\mathrm{kw}$-closed sets are the same
as sequentially $\mathrm{w}$-closed sets \cite[2.2 (a) and (b)]%
{GonzalezGutierrez92}; this follows from the Eberlein-Smulian Theorem, for
which see, e.g., \cite[Theorem 6.34]{Aliprantis-BorderIDA06}, \cite[19.4]%
{Aliprantis-BurkinshawLSRS}, \cite[2.15]{Aliprantis-BurkinshawLSRSwAE} or %
\cite[V.6.1]{Dunford-SchwartzLOI}. (So $\mathrm{kw}$ is always weaker than
the norm topology.) Another equivalent definition of $\mathrm{kw}\left(
Y,Y^{\ast }\right) $ is as the strongest topology having the same convergent
sequences as $\mathrm{w}\left( Y,Y^{\ast }\right) $; the equivalence can be
shown by using \cite[2.2 and 2.3 (b)]{GonzalezGutierrez92}.

The ``compacting'' of the weak topology fails, however, to produce a linear
one, except when it results in either $\mathrm{bw}^{\ast }$ or the norm
topology $\mathrm{m}\left( Y,Y^{\ast }\right) $. As is noted in \cite[p. 371]%
{GonzalezGutierrez92}, $\mathrm{kw}\left( Y,Y^{\ast }\right) $ is always 
\textsl{semi\/}-linear (like every $\mathrm{b}\mathcal{T}$) by \cite[Theorem
5]{Collins55}. But if $\mathrm{kw}$ is linear then it is even locally
convex, and it is so if and \textsl{only\/} if $Y$ is either (i)~a reflexive
space (in which case $\mathrm{kw}\allowbreak =\mathrm{bw}\allowbreak =%
\mathrm{bw}^{\ast }$) or (ii)~an infinite-dimensional Schur space, i.e., a
Banach space in which weakly convergent sequences are norm-convergent (in
which case $\mathrm{kw}\allowbreak =\mathrm{m}\left( Y,Y^{\ast }\right) $,
but $\mathrm{bw}$ is not linear because $Y$ is then nonreflexive): see %
\cite[2.9 and 2.5]{GonzalezGutierrez92}.\footnote{%
An infinite-dimensional Banach space cannot be both reflexive and a Schur
space; this can be seen from \cite{Orno91} and \cite[p. 2411, consequence II]%
{Rosenthal74}.} In the other cases, $\mathrm{kw}\left( Y,Y^{\ast }\right) $
is therefore different from the \textsl{convex compact weak\/} topology.
Denoted by $\mathrm{ckw}\left( Y,Y^{\ast }\right) $ or $\mathrm{ckw}$ for
brevity, this is defined as the strongest locally convex topology that is
equal to $\mathrm{w}\left( Y,Y^{\ast }\right) $ on $\mathrm{w}$-compact sets
(which is the same as the strongest locally convex topology that is weaker
than $\mathrm{kw}$).\footnote{%
Both $\mathrm{kw}$ and $\mathrm{ckw}$ can also be defined as, respectively,
the strongest topology and the strongest locally convex topology with the
same compacts as $\mathrm{w}\left( Y,Y^{\ast }\right) $.}

Furthermore, if $Y$ is reflexive then, \textsl{a fortiori\/}, it is
sequentially reflexive (i.e., $\mathrm{m}\left( Y^{\ast },Y\right) $%
-convergent sequences, in $Y^{\ast }$, are the same as the norm-convergent
a.k.a.\ $\mathrm{m}\left( Y^{\ast },Y^{\ast \ast }\right) $-convergent ones)
or, equivalently by \cite{Orno91}, $Y$ does not contain an isomorphic copy
of $l^{1}$ (i.e., no subspace of $Y$ is linearly homeomorphic to $l^{1}$).
By contrast, if $Y$ is an infinite-dimensional Schur space then it is 
\textsl{not\/} sequentially reflexive (i.e., contains $l^{1}$): see \cite[p.
2411, consequence II]{Rosenthal74}. This dichotomy corresponds exactly to
equality or inequality of $\mathrm{kw}$ and $\mathrm{bw}$, i.e., $\mathrm{cbw%
}\allowbreak =\mathrm{ckw}$ if and \textsl{only\/} if $\mathrm{bw}%
\allowbreak =\mathrm{kw}$, which is the case if and only if $Y$ is
sequentially reflexive (i.e., does not contain $l^{1}$): see \cite[2.8 and
3.3]{GonzalezGutierrez92}.

In sum, there are four---mutually exclusive and collectively
exhaustive---cases of strict inclusions and equalities (for the topologies
as families of open/closed sets):

\begin{enumerate}
\item If $Y$ is reflexive then $\mathrm{cbw}=\mathrm{bw}=\mathrm{kw}=\mathrm{%
ckw}$ (and all four are equal to $\mathrm{bw}^{\ast }$).

\item If $Y$ is not reflexive but is sequentially reflexive (or,
equivalently, does not contain $l^{1}$ as an isomorphic copy) then $\mathrm{%
ckw}\allowbreak =\mathrm{cbw}\allowbreak \varsubsetneq \mathrm{bw}%
\allowbreak =\mathrm{kw}$.

\item If $Y$ is an infinite-dimensional Schur space (and hence contains $%
l^{1}$, i.e., is not sequentially reflexive) then $\mathrm{cbw}\allowbreak
\varsubsetneq \mathrm{bw}\allowbreak \varsubsetneq \mathrm{kw}\allowbreak =%
\mathrm{ckw}$ (since $\mathrm{kw}$ is then equal to the norm topology $%
\mathrm{m}\left( Y,Y^{\ast }\right) $).

\item If $Y$ is not a Schur space but contains $l^{1}$ (i.e., is not
sequentially reflexive) then $\mathrm{cbw}\allowbreak \varsubsetneq \mathrm{%
ckw}\allowbreak \varsubsetneq \mathrm{kw}$ and $\mathrm{cbw}\allowbreak
\varsubsetneq \mathrm{bw}\allowbreak \varsubsetneq \mathrm{kw}$. That is, $%
\mathrm{ckw}$ and $\mathrm{bw}$ are two topologies that both lie strictly
between $\mathrm{cbw}$ and $\mathrm{kw}$ but are different from each other ($%
\mathrm{ckw}$ is locally convex, $\mathrm{bw}$ is not even linear)---and so
all four topologies are different. In this case $\mathrm{ckw}$ is a new
topology, i.e., it is different from all the others ($\mathrm{w}$, $\mathrm{%
cbw}$, $\mathrm{bw}$, $\mathrm{kw}$ and $\mathrm{m}\left( Y,Y^{\ast }\right) 
$).
\end{enumerate}

In this context it is worth noting that if a Banach space, $P$, contains any
infinite-dimensional Schur space, then it contains also the specific Schur
space $l^{1}$ and, furthermore, so does $P^{\ast }$; it obviously follows
that if $P$ contains $l^{1}$ then so does $P^{\ast }$ \cite[Corollaries 9
and 10]{Mujica03}. In addition, as is shown in \cite[Theorem 3]%
{PetheThakare78} and noted also in \cite[p. 371]{Mujica03}, if $P$ contains $%
l^{1}$ then $P^{\ast }$ is not a Schur space---and so, for every $%
n\allowbreak \geq 1$, the $n$\nolinebreak -$\nolinebreak $th norm-dual of $P$
contains $l^{1}$ but is not a Schur space.

\end{document}